\documentclass[12pt, twosided]{article}
\usepackage{latexsym,enumerate}
\usepackage{amsmath}
\usepackage{amsthm,amsopn}
\usepackage{amstext,amscd,amsfonts,amssymb,fontenc,bbm}
\usepackage[ansinew]{inputenc}
\usepackage{verbatim}
\usepackage{graphicx}

\usepackage{pstricks,pst-node}
\usepackage{mathtools}
\usepackage{tikz}
\usepackage{pgf,tikz}
   \usepackage{epsfig}

\usepackage{lineno} 
\usepackage{textcomp}

\newtheorem{theorem}{Theorem}[section]
\newtheorem{corollary}[theorem]{Corollary}
\newtheorem{claim}[theorem]{Claim}

\newtheorem{proposition}[theorem]{Proposition}

\theoremstyle{definition}

\begin{document}

\title{Hamiltonicity of token graphs of fan graphs}

\author{Luis Manuel Rivera\thanks{Unidad Acad\'emica de Matem\'aticas, Universidad Aut\'onoma de Zacatecas, Calzada Solidaridad entronque Paseo a la Bufa, Zacatecas, Zac. CP. 98000, Mexico.
Email: luismanuel.rivera@gmail.com} 
\and Ana Laura Trujillo-Negrete
\thanks{Departamento de Matem\'aticas, Cinvestav, Av. IPN \# 2508, Col. San Pedro Zacatenco, Mexico, Cd. de Mexico, CP. 07360, Mexico. Email: ltrujillo@math.cinvestav.mx}
}
\date{}

\maketitle{}

\begin{abstract}

 In this note we show that the token graphs of fan graphs are Hamiltonian. This result provides another proof of the Hamiltonicity of Johnson graphs and also extends previous results obtained by Mirajkar and Priyanka on the token graphs of wheel graphs.
\end{abstract}


\section{Introduction.}

Let $G$ be a simple graph of order $n$ and let $k$ be an integer such that $1\leq k\leq n-1$. The {\it $k$-token graph}, or {\it  symmetric $k$th power}, of $G$ is the graph $G^{(k)}$ whose vertices are the $k$-subsets of $V(G)$ and two vertices are adjacent in $G^{(k)}$ if their symmetric difference is an edge of $G$. A classical example is the Johnson graph $J(n, k)$, which is isomorphic to the $k$-token graph of the complete graph $K_n$. This class of graphs is widely studied and has connections with coding theory \cite{Jo, EB,GS, lieb, neun} (another connection of token graphs with coding theory was showed in \cite{oeisa}).

The definition of $k$-token graphs (without a name) appeared in a work of  Rudolph \cite{rudolph}, in connection with problems in quantum mechanics and with the graph isomorphism problem. Rudolph presented examples of cospectral graphs $G$ and $H$ such that their corresponding $2$-token graphs are not cospectral.  Audenaert et al. \cite{aude}, proved that the $2$-token graphs of strongly regular graphs with the same parameters are cospectral, and suggested that for a given positive integer $k$ there exists infinitely many pairs of non-isomorphic graphs with cospectral $k$-token graphs. This conjecture was proved by Barghi and Ponomarenko \cite{barghi} and, independently, by Alzaga et al. \cite{alzaga}. Later, Fabila-Monroy et al. \cite{FFHH} reintroduced the $k$-token graphs as part of several models of swapping in the literature \cite{vanden,yama}, and studied some properties of these graphs: connectivity, diameter, cliques, chromatic number, Hamiltonian paths and Cartesian product of token graphs. This line of research was continued by Carballosa et al. \cite{token2} who studied regularity and planarity, de Alba et al. \cite{dealba2}, who presented some results about independence and matching numbers, and Mirajkar et al.~\cite{keer}, who studied some covering properties of token graphs. Finally, Lea\~nos and Trujillo-Negrete~\cite{leatrujillo} proved a conjecture of Fabila-Monroy et al. \cite{FFHH} about the connectivity of token graphs and de Alba et. al. \cite{dealba} classified the triangular graphs (in other words, the $2$-token graphs of complete graphs) that are Cohen-Macaulay. 

A graph is Hamiltonian if it contains a Hamiltonian cycle. It is well known that $J(n, k)$ is Hamiltonian~\cite{ho, zhang}, in fact, it is Hamiltonian connected~\cite{asplach}. As was noted in~\cite{FFHH}, the existence of a Hamiltonian cycle in $G$ does not imply that $G^{(k)}$ contains a Hamiltonian cycle. For example, if $k$ is even then $K_{m, m}^{(k)}$ is not Hamiltonian. 

We are interested in graphs $G$ such that its token graphs are Hamiltonian. The fan graph $F_{n}$ is the join of graphs $K_1$ and $P_{n-1}$. In this note we show that the token graphs of fan graphs are Hamiltonian. Our result provides another proof that $J(n, k)$ is Hamiltonian, and also extends some of  the results obtained by Mirajkar and Priyanca Y. B.~\cite{keer} about the Hamiltonicity of the token graphs of wheel graphs. 

\section{Main result}
First we present some definitions and notations. For vertices $u, v$ in graph $G$ we write $u \sim v$ to mean that $u$ and $v$ are adjacent vertices in $G$. We write $G \simeq G'$ if $G$ and $G'$ are isomorphic graphs. A spanning subgraph of $G$ is a subgraph $H$ such that $V(H)=V(G)$. The following proposition is obvious.  
\begin{proposition}\label{hspanning}
If $H$ is a spanning subgraph of $G$ and $H$ is Hamiltonian then $G$ is Hamiltonian.
\end{proposition}
One of the main properties of token graphs is that $G^{(1)}$ and $G$ are isomorphic. Moreover,  $G^{(k)} \simeq G^{(n-k)}$ for any $k \in \{1, \dots, n-1\}$. Another known property of token graphs is the following. 

\begin{proposition}\label{tokenspanning}
If $H$ is a subgraph of $G$ then $H^{(k)}$ is a subgraph of $G^{(k)}$. Even more, if $H$ is a spanning subgraph of $G$ then  $H^{(k)}$ is a spanning subgraph of $G^{(k)}$.
\end{proposition}
 
For a fan graph $F_n$ we assume that the vertices of $P_{n-1}$ are $\{1, \dots, n-1\}$ and the vertex in $K_1$ is labeled as $n$. For vertex $A=\{a_1, \dots, a_k\}$ of $F_n^{(k)}$ we use the convention that $a_1< \dots < a_k$. 

The main result of this note is the following.

\begin{theorem}\label{mainth}
Let $n$ and $k$ be positive integers with $n\geq 3$ and $1\leq k \leq n-1$. Then $F_n^{\{k\}}$ is Hamiltonian.
\end{theorem}
\begin{proof}
For $k=1$, $F_n^{(1)} \simeq F_n$ which is Hamiltonian so in the rest of the proof we assume that $k\geq 2$. We will show that $F_n^{(k)}$ has a Hamiltonian cycle such that the vertices $\{n-k, n-k+1,\dots,n-2, n\}$  and  $\{n-k, n-k+1,\dots, n-2, n-1\}$ are adjacent in the cycle. The sequence of vertices $\{1, 3\}\{1, 2\}\{2,3\}\{1, 3\}$ is a Hamiltonian cycle in $F_3^{(2)}$.
The proof for $n \geq 4$ is by double induction. First we show the case $k=2$ and $n \geq 4$. The sequence of vertices 
\begin{equation*}
\begin{split}
&\{1,n-1\}\{1,n\}\{1,n-2\}\{1,n-3\}\dots\{1,3\}\{1,2\} \\ &\{2,n\}\{2,n-1\}\{2,n-2\}\{2,n-3\}\dots\{2,4\}\{2,3\} \\ & \qquad \qquad \qquad \qquad \qquad \vdots \\ &\{n-3,n\}\{n-3,n-1\}\{n-3,n-2\}\\ &\{n-2,n\}\{n-2,n-1\}\\&\{n-1,n\}\\&\{1,n-1\},
\end{split}
\end{equation*}
 is a Hamiltonian cycle in $F_n^{(2)}$, where vertices $\{n-2,n-1\}$ and $\{n-2,n\}$ are adjacent in the cycle. We assume as induction hypothesis that $F_{n'}^{(k^\prime)}$ satisfies the conditions whenever $k^\prime<k$ and $n^\prime > k^\prime$, or $F_{n'}^{(k)}$ satisfies the conditions whenever $n^\prime < n$ and $n^\prime > k$.
 \begin{claim}
Let $S_i$ be the subgraph of $F_n^{(k)}$ induced by the vertex set  
\[
V_i=\left\{ \{a_1, \dots, a_k\} \in V(F_n^{(k)}) \colon a_1=i\right\},
\]
with $1\leq i\leq n-k$. Then $S_i \simeq F_{n-i}^{(k-1)}$. 
\end{claim}
\begin{proof} Suppose that $V(F_{n-i})=\{i+1, \dots, n\}$ with $V(P_{n-i-1})=\{i+1, \dots, n-1\}$ and $n$ the vertex of $K_1$. Then the function $A \mapsto A\setminus\{i\}$ is a graph isomorphism between $S_i$ and $F_{n-i}^{(k-1)}$.
\end{proof}

We identify $S_i$ with $F_{n-i}^{(k-1)}$ using the isomorphism given in the proof of the claim. By induction there exists a Hamiltonian cycle $C_i$ in $S_i$, where vertices $X_i:=\{i, n-k+1, \dots, n-2,n-1\}$ and $Y_i:=\{i, n-k+1, \dots, n-2,n\}$ are adjacent in $C_i$, for $1\leq i\leq n-k$. Let $P_i$ be the Hamiltonian subpath of $C_i$ from $X_i$ to $Y_i$, for $1\leq i\leq n-k$. Let $Z$ denote the vertex $\{n-k+1, n-k+2, \dots, n-1, n\}$. Therefore $V_{n-k+1}=\{Z\}$. 

Let $D_{i}=\{n-k,n-k+1, \dots, n-1, n\}\setminus \{i\}$, with $n-k+1\leq i \leq n$. Then the vertes set $V_{n-k}$ of $S_{n-k}$ is $\{D_n, D_{n-1},\dots, D_{n-k+1}\}$. Also, we have $X_{n-k}=D_n$ and  $Y_{n-k}=D_{n-1}$. Let 
\begin{equation*}
Q=D_{n-2}D_{n-3} \dots D_{n-k+2}D_{n-k+1},
\end{equation*}
which, in fact, is a path in $S_{n-k}$ because $D_i\triangle D_{i-1}=\{i-1,i\}$,
for $n-k+2\leq i\leq n-2$. Now,
\begin{eqnarray*}
X_{n-k}\triangle D_{n-2}&=&\{n-2,n\} \\ 
Y_{n-k}\triangle D_{n-2}&=&\{n-2,n-1\} \\ 
Z\triangle D_{n-k+1}&=&\{n-k,n-k+1\}
\end{eqnarray*}
and hence
\begin{equation*}
\label{eq:sim}
\begin{split}
& X_{n-k}\sim D_{n-2}, \\
&Y_{n-k}\sim D_{n-2}, \\
&D_{n-k+1}\sim Z, 
\end{split}
\end{equation*} 
in $F_n^{(k)}$. Notice that $X_i\triangle X_{i+1}=\{i,i+1\}$ and $Y_i\triangle Y_{i+1}=\{i,i+1\}$, for $1\leq i\leq n-k-1$, and $X_1\triangle Z=\{1,n\}$. Therefore we can define a Hamiltonian cycle $\mathcal{C}$ in $F_n^{(k)}$ as 
\[
X_1\stackrel{\longrightarrow}{\text{\tiny $P_1$}}Y_1Y_2\stackrel{\longrightarrow}{\text{\tiny $P_2$}}X_2\dots X_{(n-k-1)}\stackrel{\longrightarrow}{\text{\tiny $P_{n-k-1}$}}Y_{(n-k-1)}Y_{(n-k)}X_{(n-k)}D_{n-2}\stackrel{\longrightarrow}{\text{\tiny $Q$}} D_{n-k+1}ZX_1,
\]
if $n-k$ is even, and 
\[
X_1\stackrel{\longrightarrow}{\text{\tiny $P_1$}}Y_1Y_2\stackrel{\longrightarrow}{\text{\tiny $P_2$}}X_2\dots Y_{(n-k-1)}\stackrel{\longrightarrow}{\text{\tiny $P_{n-k-1}$}}X_{(n-k-1)}X_{(n-k)}Y_{(n-k)}D_{n-2}\stackrel{\longrightarrow}{\text{\tiny $Q$}}D_{n-k+1}ZX_1,
\]
if $n-k$ is odd.
 Furthermore  
\[
\{n-k,n-k+1,...,n-2,n-1\}=X_{n-k} \sim Y_{n-k}=\{n-k,n-k+1,...,n-2,n\},
\] in $\mathcal{C}$, as desired. 

\end{proof}

The wheel graph $W_n$ is the joint graph of $K_1$ and $C_{n-1}$. It is known that Johnson graphs~\cite{ho, zhang} and the $k$-token graphs of wheel graphs~\cite{keer} are Hamiltonian, the following corollary provides another proof of this facts.
\begin{corollary}
If $F_n$ is a spanning subgraph of $G$ then $G^{(k)}$ is Hamiltonian. In particular the Johnson graphs and the $k$-token graphs of wheel graphs are Hamiltonian.
\end{corollary}
\begin{proof}
As $F_n$ is a spanning subgraph of $G$ then $H^{(k)}$ is a spanning subgraph of $G^{(k)}$ by Proposition~\ref{tokenspanning}. The $k$-token graph of $F_n$ is Hamiltonian by Theorem~\ref{mainth}  and hence $G^{(k)}$ is Hamiltonian by Proposition~\ref{hspanning}.  In particular $F_n$ is a spanning subgraph of $W_n$ and $K_n$.
\end{proof}
\section*{Acknowledgments}
Part of this work was made when the second author was a master student at Universidad Aut\'onoma de Zacatecas. The authors would like to thank the anonymous reviewer for his/her corrections and suggestions.

\end{document}